\documentclass[11pt,psamsfonts]{amsart}
\usepackage{amsmath}
\usepackage{amsthm}
\usepackage{amssymb}
\usepackage{amscd}
\usepackage{amsfonts}
\usepackage{amsbsy}
\usepackage{epsfig,afterpage}
\usepackage{psfrag}
\usepackage{graphicx}
\usepackage{indentfirst, latexsym, bm,amssymb}
\usepackage{bbding}

\newtheorem {theorem} {Theorem}

\newtheorem {lemma}  [theorem]{Lemma}

\begin{document}

\title[Generalized rational first integrals]
{Generalized rational first integrals\\
of analytic differential systems}

\author[Wang Cong, Jaume Llibre and Xiang Zhang]
{Wang Cong$^1$, Jaume Llibre$^2$ and Xiang Zhang$^1$}

\address{$^1$ Department of Mathematics, Shanghai Jiaotong University, Shanghai,
200240, P. R. China}
\email{wangcong86@sjtu.edu.cn, xzhang@sjtu.edu.cn}

\address{$^2$ Departament de Matem\`{a}tiques, Universitat Aut\`{o}noma de Barcelona,
08193 Bellaterra, Barcelona, Catalonia, Spain}
\email{jllibre@mat.uab.cat}



\date{}

\begin{abstract}
In this paper we mainly study the necessary conditions for the
existence of functionally independent generalized rational first
integrals of ordinary differential systems via the resonances. The
main results extend some of the previous related ones, for instance
the classical Poincar\'e's one \cite{Po}, the Furta's one \cite{Fu},
part of Chen {\it et al}'s ones \cite{CYZ}, and the Shi's one
\cite{Sh}. The key point in the proof of our main results is that
functionally independence of generalized rational functions implies
the functionally independence of their lowest order rational
homogeneous terms.
\end{abstract}

\maketitle

\section{Introduction and statement of the main results}\label{s1}

The rational first integrals in analytic differentiable systems and
mainly in the particular case of polynomial differentiable systems
has been studied intensively, specially inside the Darboux theory of
integrability, see for instance \cite{Ll}, \cite{GGG}, \cite{PS},
\cite{Si}. In this paper we want to study the generalized rational
first integrals of the analytic differential systems.

\smallskip

Consider analytic differential systems in $(\mathbb C^n,0)$
\begin{equation}\label{e1.1}
\dot x=f(x),\qquad x\in (\mathbb C^n,0).
\end{equation}
A function $F(x)$ of form $G(x)/H(x)$ with $G$ and $H$ analytic
functions in $(\mathbb C^n,0)$ is a {\it generalized rational first
integral} if
\[
\left\langle f(x),\,\partial_x F(x)\right\rangle\equiv 0, \qquad x\in (\mathbb C^n,0),
\]
where $\langle\cdot,\cdot\rangle$ denotes the inner product of two
vectors in $\mathbb C^n$,  and
$\partial_xF=(\partial_{x_1}F,\ldots,\partial_{x_n}F)$ is the
gradient of $F$ and $\partial_{x_i}F=\partial F/\partial {x_i}$. As
usually, if $G$ and $H$ are polynomial functions, then $F(x)$ is a
{\it rational first integral}. If $H$ is a non--zero constant, then
$F(x)$ is  an {\it analytic first integral}. So generalized rational
first integrals include rational first integrals and analytic first
integrals as particular cases.

\smallskip

If $f(0)\ne 0$, it is well--known from the Flow--Box Theorem that
system \eqref{e1.1} has an analytic first integral in a neighborhood
of the origin. If $f(0)=0$, i.e. $x=0$ is a {\it singularity} of
system \eqref{e1.1}, the existence of first integrals for system
\eqref{e1.1} in $(\mathbb C^n,0)$ is usually much involved.

\smallskip

Denote by $A=Df(0)$ the Jacobian matrix of $f(x)$ at $x=0$. Let
$\lambda=(\lambda_1,\ldots,\lambda_n)$ be the $n$--tuple of
eigenvalues of $A$. We say that the eigenvalues $\lambda$ satisfy a
$\mathbb Z^+$--{\it resonant condition} if
\[
\left\langle \lambda, \mathbf k\right\rangle=0, \qquad\mbox{ for some
} \mathbf  k\in \left(\mathbb Z^+\right)^n,\quad \mathbf k\ne 0,
\]
where $\mathbb Z^+=\mathbb N\cup\{0\}$ and $\mathbb N$ is the set of
positive integers. The eigenvalues $\lambda$ satisfy a $\mathbb
Z$--{\it resonant condition} if
\[
\left\langle \lambda, \mathbf k\right\rangle=0, \qquad\mbox{ for some
} \mathbf  k\in \mathbb N^n,\quad \mathbf k\ne 0,
\]
where $\mathbb Z$ is the set of integers.

\smallskip

Poincar\'e \cite{Po} was the first one in studying the relation
between the existence of analytic first integrals and resonance, he
obtained the following classical result (for a proof, see for
instance \cite{Fu}).

\smallskip

\noindent{\bf Poincar\'e Theorem} {\it If the eigenvalues $\lambda$
of $A$ do not satisfy any $\mathbb Z^+$--resonant conditions, then
system \eqref{e1.1} has no analytic first integrals in $(\mathbb
C^n,0)$.}

\smallskip

Recall that $k$ ($k< n$) functions are {\it functionally independent}
in an open subset $U$ of $\mathbb C^n$ if their gradients have rank
$k$ in a full Lebesgue measure subset of $U$. Obviously an
$n$--dimensional nontrivial autonomous system can have at most $n-1$
functionally independent first integrals, where nontrivial means that
the associated vector field does not vanish identically.

\smallskip

In 2003 Li, Llibre and Zhang \cite{LLZ03} extended the Poincar\'e's
result to the case that $\lambda$ admit one zero eigenvalue and the
others are not $\mathbb Z^+$--resonant.

\smallskip

In 2008 Chen, Yi and Zhang \cite{CYZ} proved that the number of
functionally independent analytic first integrals for system
\eqref{e1.1} does not exceed the maximal number of linearly
independent elements of $\{\mathbf k\in(\mathbb Z^+)^n:\, \langle
\mathbf k, \lambda\rangle=0,\, \mathbf k\ne 0\}$.

\smallskip

In 2007 the Poincar\'e's result was extended by Shi \cite{Sh} to the
$\mathbb Z$--resonant case. He proved that if system \eqref{e1.1} has
a rational first integral, then the eigenvalues $\lambda$ of $A$
satisfy a $\mathbb Z$--resonant condition. In other words, if
$\lambda$ do not satisfy any $\mathbb Z$--resonant condition, then
system \eqref{e1.1} has no rational first integrals in $(\mathbb
C^n,0)$.

\smallskip

The aim of this paper is to improve the above results by studying the
existence of more than one functionally independent rational first
integrals. Our first main result is the following.

\smallskip

\begin{theorem}\label{t1}
Assume that the differential system \eqref{e1.1} satisfies $f(0)=0$
and let $\lambda=(\lambda_1,\ldots,\lambda_n)$ be the eigenvalues of
$Df(0)$. Then the number of functionally independent generalized
rational first integrals of system \eqref{e1.1} in $(\mathbb C^n,0)$
is at most the dimension of the minimal vector subspace of $\mathbb
R^n$ containing the set $\{\mathbf k\in\mathbb Z^n:\, \langle \mathbf
k, \lambda\rangle=0,\,\mathbf k\ne 0\}$.
\end{theorem}

\smallskip

We remark that Theorem \ref{t1} extends all the results mentioned
above, i.e. the one of Poincar\'e \cite{Po}, the Theorem 1.1 of Chen,
Yi and Zhang \cite{CYZ}, and  the Theorem 1 of Shi \cite{Sh}. We
should mention that the methods of the above mentioned papers are not
enough to study the existence of more than one functionally
independent generalized rational first integrals. Here we will use a
different approach to prove Theorem \ref{t1}. Our key technique will
be the Lemma \ref{l3.2} given in Section \ref{s3}, which shows that
the functionally independence of generalized rational functions
implies the functionally independence of their lowest order rational
homogeneous terms.

\smallskip

We should say that Theorem \ref{t1} has some relation with
Propositions 3.5 and 5.4 of Goriely \cite{Go}. In the former the
author established a relation between the weight degrees of
independent algebraic first integrals of a weight homogeneous vector
field and its Kowalevskaya exponents. And in the latter he provided a
necessary condition for the existence of independent analytic first
integrals of a weight homogeneous vector field.

\smallskip

We note that if the linear part $Df(0)$ of \eqref{e1.1} has all its
eigenvalues zero, then the result of Theorem \ref{t1} is trivial. For
studying these cases we consider  semi--quasi--homogeneous systems.
Let $f=(f_1,\ldots,f_n)$ be vector functions. Then system
\eqref{e1.1} is {\it quasi--homogeneous of degree $q\in \mathbb
N\setminus\{1\}$} with exponents $s_1,\ldots,s_n\in\mathbb
Z\setminus\{0\}$ if for all $\rho>0$
\[
f_i\left(\rho^{s_1}x_1,\ldots,\rho^{s_n}x_n\right)=\rho^{q+s_i-1}
f_i(x_1,\ldots,x_n),\quad i=1,\ldots,n.
\]
The exponents $\mathbf s:=(s_1,\ldots,s_n)$ are called {\it weight
exponents}, and the number $q+s_i-1$ is called the {\it weight
degree} of $f_i$, i.e., $f_i$ is a {\it quasi--homogeneous function
of weight degree $q+s_i-1$}. A vector function $f$ is {\it
quasi--homogeneous of weight degree $q$ with weight exponent}$\mathbf
s$ if each component $f_i$ is quasi--homogeneous of weight degree
$q$, i.e. $f_i\left(\rho^{s_1}x_1,\ldots,\rho^{s_n}x_n\right)=
\rho^{q}f_i(x_1,\ldots,x_n)$ for $i=1,\ldots,n$.

\smallskip

System \eqref{e1.1} is {\it semi--quasi--homogeneous} of degree $q$
with the weight exponent $\mathbf s$ if
\begin{equation}\label{e1.5}
f(x)=f_q(x)+f_h(x),
\end{equation}
where $\rho^{\mathbf E-\mathbf S}f_q$ is quasi--homogeneous of degree
$q$ with weight exponent $\mathbf s$ and $\rho^{\mathbf E-\mathbf
S}f_h$ is the sum of quasi--homogeneous of degree either all larger
than $q$ or all less than $q$  with weight exponent $\mathbf s$. The
former (resp. latter) is called {\it positively} (resp. {\it
negatively}) semi--quasi--homogeneous. Here we have used the
notations: $\mathbf E$ is the $n\times n$ identity matrix, $\mathbf
S$ is the $n\times n$ diagonal matrix $\mbox{diag}(s_1,\ldots,s_n)$,
and $\rho^{\mathbf E-\mathbf S}=\mbox{diag}\left(\rho^{1-s_1},\ldots,
\rho^{1-s_n}\right)$.

\smallskip

We note that for any give exponent $\mathbf s$ an analytic
differential system \eqref{e1.1} can be written as a positively
semi--quasi--homogeneous system, and also as a negatively one if it
is a polynomial.

\smallskip

Assume that $\rho^{\mathbf E-\mathbf S}f_q$ is quasi--homogeneous of
weight degree $q>1$ with weight exponent $\mathbf s$. Set $\mathbf
W=\mathbf S/(q-1)$. Any solution $c=(c_1,\ldots,c_n)$ of
\begin{equation}\label{e1.6}
f_q(c)+\mathbf W c=0,
\end{equation}
is called a {\it balance}. Denote by $\mathcal B$ the set of
balances. For each balance $c$, the Jacobian matrix $K=D
f_q(c)+\mathbf W$ is called the {\it Kowalevskaya matrix} at $c$ and
its eigenvalues are called the {\it Kowalevskaya exponents}, denoted
by $\lambda_c$. Let $d_c$ be the dimension of the minimal vector
subspace of $\mathbb R^n$ containing the set
\[
\left\{\mathbf k\in \mathbb Z^n:\, \langle \mathbf
k,\lambda_c\rangle=0,\,\mathbf k\ne 0\right\}.
\]

\smallskip

\begin{theorem}\label{t4}
Assume that system \eqref{e1.1} is semi--quasi--homogeneous of weight
degree $q$ with weight exponent $\mathbf s$, and $f(x)$ satisfies
\eqref{e1.5} with $f(0)=0$. Then the number of functionally
independent generalized rational first integrals of \eqref{e1.1} is
at most $d=\min\limits_{c\in \mathcal B}d_c$.
\end{theorem}

\smallskip

Theorem \ref{t4} is an extension of Theorem 1 of Furta \cite{Fu}, of
Corollary 3.7 of \cite{Go} and of Theorem 2 of Shi \cite{Sh}. In some
sense it is also an extension of the main results of Yoshida
\cite{Yo1,Yo2}, where he proved that if a quasi--homogenous
differential system is algebraically integrable, then every
Kowalevkaya exponent should be a rational number.

\smallskip

Theorems \ref{t1} and \ref{t4} studied the existence of functionally
independent generalized rational first integrals of system
\eqref{e1.1} in a neighborhood of a singularity. Now we turn to
investigate the existence of generalized rational first integrals of
system \eqref{e1.1} in a neighborhood of a periodic orbit. The {\it
multipliers} of a periodic orbit are the eigenvalues of the linear
part of the Poincar\'e map at the fixed point corresponding to the
periodic orbit. Recall that a Poincar\'e map associated to a periodic
orbit is defined on a transversal section to the periodic orbit, and
its linear part has the eigenvalue $1$ along the direction tangent to
the periodic orbit.

\smallskip

\begin{theorem}\label{t5}
Assume that the analytic differential system \eqref{e1.1} has a
periodic orbit with multipliers $\mu=(\mu_1,\dots,\mu_{n-1})$. Then
the number of functionally independent generalized rational first
integrals of system \eqref{e1.1} in a neighborhood of the periodic
orbit is at most the maximum number of linearly independent vectors
in $\mathbb R^n$ of the set
\[
\{\mathbf k\in \mathbb Z^{n-1}:\, \mu^{\mathbf k}=1,\,\mathbf k\ne 0\}.
\]
\end{theorem}

\smallskip

Here and after, for vectors $x=(x_1,\ldots,x_n)$ and $\mathbf
k=(k_1,\ldots,k_n)$ we use $x^{\mathbf k}$ to denote the product
$x_1^{k_1}\cdot\ldots\cdot x_n^{k_n}$.

\smallskip

Finally we consider the periodic differential systems
\begin{equation}\label{e1.2}
\dot x=f(t,x),\qquad (t,x)\in \mathbb S^1\times (\mathbb C^n,0),
\end{equation}
where $\mathbb S^1=\mathbb R/(2\pi \mathbb N)$, and $f(t,x)$ is
analytic in its variables and periodic of period $2\pi$ in $t$.
Assume that $x=0$ is a constant solution of \eqref{e1.2}, i.e.
$f(t,0)=0$.

\smallskip

A non--constant function $F(t,x)$ is a {\it generalized rational
first integral} of system \eqref{e1.2} if $F(t,x)=G(t,x)/H(t,x)$ with
$G(t,x)$ and $H(t,x)$ analytic in their variables and $2\pi$ periodic
in $t$, and it satisfies
\[
\frac{\partial F(t,x)}{\partial t}+\langle
\partial_xF(t,x),f(t,x)\rangle\equiv 0 \quad \mbox{ in }  \mathbb
S^1\times (\mathbb C^n,0).
\]
If $H(t,x)$is constant and non--zero, then $F(t,x)$ is an {\it
analytic first integral}. If $G(t,x)$ and $H(t,x)$ are polynomials in
$x$ we say that $F(t,x)$ is a {\it rational first integral}. If
$G(t,x)$ and $H(t,x)$ are both homogeneous polynomials in $x$ of
degrees $l$ and $m$ respectively, we say that $G/H$ is a {\it
rational homogeneous first integral of degree} $l-m$.

\smallskip

Since $f(t,0)=0$, we can write system \eqref{e1.2} as
\begin{equation}\label{e1.3}
\dot x=A(t)x+g(t,x),
\end{equation}
where $A(t)$ and $g(t,x)=O(x^2)$ are $2\pi$ periodic in $t$. Consider
the linear equation
\begin{equation}\label{e1.4}
\dot x=A(t)x.
\end{equation}
Let $x_0(t)$ be the solution of \eqref{e1.4} satisfying the initial
condition $x_0(0)=x_0$ with $x_0\in (\mathbb C^n,0)$. The {\it
monodromy operator} associated to \eqref{e1.4} is the map $\mathcal
P:\,(\mathbb C^{n-1},0)\rightarrow (\mathbb C^{n-1},0)$ defined by
$\mathcal P (x_0)=x_0(2\pi)$.

\smallskip

We say that functions $F_1(t,x),\ldots,F_m(t,x)$ are {\it
functionally independent} in $\mathbb S^1\times (\mathbb C^n,0)$ if
$\partial_xF_1(t,x),\ldots,\partial_xF_m(t,x)$ have the rank $m$ in a
full Lebesgue measure subset of $\mathbb S^1\times (\mathbb C^n,0)$.

\begin{theorem}\label{t3}
Let $\mu=(\mu_1,\ldots,\mu_{n})$ be the eigenvalues of the
monodromy operator $($i.e. the {\it characteristic multipliers} of
\eqref{e1.4}$)$. Then the number of functionally independent
generalized rational first integrals of system \eqref{e1.2} is at
most the maximum number of linearly independent vectors in $\mathbb
R^n$ of the set
\[
\Xi:=\left\{\mathbf k\in\mathbb Z^n:\, \mu^{\mathbf k}=1,\mathbf k\ne
0\right\}\subset \mathbb Z^n.
\]
\end{theorem}

We remark that Theorem \ref{t3} is an improvement of Theorem $5$ of
\cite{LLZ03} in two ways: one is from analytic (formal) first
integrals to generalized rational first integrals, and second our
result is for more than one first integral.

\smallskip

If there exits a non--zero $\mathbf k\in\mathbb Z^n$ such that
$\mu^{\mathbf k}=1$, we say that $\mu$ is {\it resonant}. The set
$\Xi$ in Theorem \ref{t3} is called the {\it resonant lattice}. For a
$\mathbf k\in\Xi$, we say that $y^{\mathbf k}$ is a {\it resonant
monomial}.

\smallskip

The rest of this paper is dedicated to prove our main results. The
proof of Theorems \ref{t1}, \ref{t4}, \ref{t5} and \ref{t3} is
presented in sections \ref{s3}, \ref{s5}, \ref{s6} and \ref{s4},
respectively.

\section{Proof of Theorem \ref{t1}} \label{s3}

Before proving Theorem \ref{t1} we need some preliminaries, which are
the heart of the proof of Theorem \ref{t1}.

Let $\mathbb C(x)$ be the field of rational functions in the
variables of $x$, and $\mathbb C[x]$ be the ring of polynomials in
$x$. We say that the functions $F_1(x), \ldots, F_k(x)\in \mathbb
C(x)$ are {\it algebraically dependent} if there exists a complex
polynomial $P$ of $k$ variables such that
$P(F_1(x),\ldots,F_k(x))\equiv 0$. For general definition on
algebraical dependence in a more general field, see for instance
\cite[p.152]{Ful}).

\smallskip

The first result provides an equivalent condition on functional
independence. The main idea of the proof follows from that of Ito
\cite[Lemma 9.1]{It89}.

\begin{lemma}\label{l3.1}
The functions $F_1(x), \ldots, F_k(x)\in \mathbb C(x)$ are
algebraically independent if and only if they are functionally
independent.
\end{lemma}

\smallskip

\begin{proof}
Sufficiency. By contradiction, if $F_1(x),\ldots,F_k(x)$ are
algebraically dependent, then there exists a complex polynomial
$P(z_1,\ldots,z_k)$ of minimal degree such that
\begin{equation}\label{e*}
P(F_1(x),\ldots,F_k(x))\equiv 0.
\end{equation}
Here minimal degree means that for any polynomial $Q(z_1,\ldots,z_k)$
of degree less than $\mbox{deg}P$ we have that
$Q(F_1(x),\ldots,F_k(x))\not\equiv 0$.

{From} \eqref{e*} it follows that $\partial_{x_j} P(F_1(x),\ldots,
F_k(x))\equiv 0$ for $j=1,\ldots,n$. These are equivalent to
\begin{equation}\label{e3.2}
\displaystyle
\left(\begin{array}{ccc}\displaystyle
\frac{\partial F_1(x)}{\partial x_1} &\displaystyle \ldots &
\displaystyle \frac{\partial F_k(x)}{\partial x_1}\\
\displaystyle\vdots &\displaystyle \ddots &\displaystyle \vdots \\
\displaystyle\frac{\partial F_1(x)}{\partial x_n} &\displaystyle
\ldots &\displaystyle \frac{\partial F_k(x)}{\partial x_n}
\end{array}\right)
\left(\begin{array}{c}
\displaystyle\frac{\partial P}{\partial z_1}(F_1(x),\ldots,F_k(x))\\
\displaystyle\vdots \\
\displaystyle\frac{\partial P}{\partial z_k}(F_1(x),\ldots,F_k(x))
\end{array}\right)\equiv 0.
\end{equation}
Since $P(z_1,\ldots,z_n)$ has the minimal degree and some of the
derivatives $\dfrac{\partial P}{\partial z_l}(z_1,\ldots,z_n)\ne 0$,
it follows that
\[
\displaystyle\frac{\partial P}{\partial z_1}(F_1(x),\ldots,F_k(x))\equiv 0,\,\, \ldots, \,\,
\displaystyle\frac{\partial P}{\partial z_k}(F_1(x),\ldots,F_k(x))\equiv 0,
\]
cannot simultaneously hold. This shows that the rank of the $n\times
k$ matrix in \eqref{e3.2} is less than $k$. Consequently
$F_1(x),\ldots, F_k(x)$ are functionally dependent, this is in
contradiction with the assumption. So we have proved that if
$F_1(x),\ldots,F_n(x)$ are functionally independent, then they are
algebraically independent.

\smallskip

Necessity. For proving this part, we will use the theory of field
extension. For any $v_1,\dots,v_r$ which are elements of some
finitely generated field extension of $\mathbb C$, we denote by
$\mathbb C(v_1,\ldots,v_r)$ the minimal field containing
$v_1,\ldots,v_r$ (for more information on finitely generated field
extensions, see for instance \cite[Chapter one, $\S$\,8]{Ful}).

\smallskip

Recall that for a finitely generated field extension $K$ of the field
$\mathbb C$, denoted by $K/\mathbb C$, the {\it transcendence degree}
of $K$ over $\mathbb C$ is defined to be the smallest integer $m$
such that for some $y_1,\ldots,y_m\in K$, $K$ is algebraic over
$\mathbb C(y_1,\ldots,y_m)$, the field of complex coefficient
rational functions in $y_1,\ldots,y_m$ (see for instance
\cite[p.149]{Ful}). We say that $K$ is {\it algebraic} over $\mathbb
C(y_1,\ldots,y_n)$ if every element, saying $k$, of $K$ is algebraic
over $\mathbb C(y_1,\ldots,y_m)$, i.e. there exists a monic
polynomial $p(X)=X^l+p_1X^{l-1}+\ldots+p_l\in\mathbb
C(y_1,\ldots,y_m)[X]$ such that $p(k)=0$, where by definition $p_j\in
\mathbb C(y_1,\ldots,y_m)$ for $j=1,\ldots,l$ (see for instance
\cite[p.29]{Ful}). The elements $\{y_1,\ldots,y_m\}$ is called a {\it
transcendence base} of $K$ over $\mathbb C$ (see for instance
\cite[p.152]{Ful}). A finitely generated field extension $K$ of
$\mathbb C$ is {\it separably generated} if there is a transcendence
base $\{z_1,\ldots,z_m\}$ of $K$ over $\mathbb C$ such that $K$ is a
separable algebraic extension of $\mathbb C(z_1,\ldots,z_m)$ (see for
instance \cite[p.27]{Ha}). Let $K$ be a field extension of a field
$L$.  The field extension $K/L$ is called {\it algebraic} if $K$ is
algebraic over $L$. An algebraic extension $K/L$ is {\it separable}
if for every $\alpha\in K$, the minimal polynomial of $\alpha$ over
$L$ is separable. A polynomial is {\it separable} over a field $L$ if
all of its irreducible factors have distinct roots in an algebraic
closure of $L$.

\smallskip

Since $F_1,\ldots,F_k$ are algebraically independent, $\mathbb
C(F_1,\ldots,F_k)$ is a separably generated and finitely generated
field extension of $\mathbb C$  of transcendence degree $k$. Here
separabililty follows from the fact that $\mathbb C$ is of
characteristic $0$, and so is $\mathbb C(F_1,\ldots,F_k)$. This last
claim follows from \cite[Theorem 4.8A]{Ha}, which states that if $k$
is an algebraically closed field, then any finitely generated field
extension $K$ of $k$ is separably generated.

\smallskip

{From} the theory of derivations over a field (see for instance
\cite[Chapter X, Theorem 10]{Sl}), there exist $k$ derivations
$D_r(r=1,\ldots,k)$ on $\mathbb
C(F_1,\ldots,F_k)$ satisfying
\begin{equation}\label{e3.3}
D_rF_s=\delta_{rs},
\end{equation}
where $\delta_{rs}=0$ if $r\ne s$, or $\delta_{rs}=1$ if $r=s$.

Since $F_1,\ldots, F_k$ are algebraically independent, it follows
that $\mathbb C(x)$ is a finitely generated field extension of
$\mathbb C(F_1,\ldots,F_k)$ of transcendence degree $n-k$. Hence
there exist $n$ derivations $\widetilde D_1,\ldots,\widetilde D_n$ on
$\mathbb C(x)$ satisfying $\widetilde D_j=D_j$ on $\mathbb
C(F_1,\ldots,F_k)$ for $j=1,\ldots,k$. In addition, all derivations
on $\mathbb C(x)$ form an $n$--dimensional vector space over $\mathbb
C(x)$ with base $\{\frac{\partial }{\partial x_j}:\,\,
j=1,\ldots,n\}$. So we have
\[
\widetilde
D_s=\sum\limits_{j=1}\limits^nd_{sj}\frac{\partial}{\partial x_j},
\]
where $d_{sj}\in \mathbb C(x)$. The derivations $\widetilde D_s$
acting on $\mathbb C(F_1,\ldots,F_k)$ satisfy
\[
\delta_{sr}=D_sF_r=\widetilde
D_sF_r=\sum\limits_{j=1}\limits^nd_{sj}\frac{\partial F_r}{\partial
x_j},\qquad r,s\in\{1,\ldots,k\}.
\]
This shows that the gradients $\nabla_x F_1,\ldots,\nabla_x F_k$ have
the rank $k$, and consequently $F_1,\ldots,F_k$ are functionally
independent.
\end{proof}

\smallskip

We remark that Lemma \ref{l3.1} has a relation in some sense with the
result of Bruns in 1887 (see \cite{Fo}), which stated that if a
polynomial differential system of dimension $n$ has $l\,(1\le l\le
n-1)$ independent algebraic first integrals, then it has $l$
independent rational first integrals. For a short proof of this
result, see Lemma 2.4 of Goriely \cite{Go}.

\smallskip

For an analytic or a polynomial function $F(x)$ in $(\mathbb C^n,0)$,
in what follows we denote by $F^0(x)$ its lowest degree homogeneous
term. For a rational or a generalized rational function
$F(x)=G(x)/H(x)$ in $(\mathbb C^n,0)$, we denote by $F^0(x)$ the
rational function $G^0(x)/H^0(x)$. We expand the analytic functions
$G(x)$ and $H(x)$ as
\[
G^0(x)+\sum\limits_{i=1}\limits^{\infty}G^i(x)\quad \mbox{ and }\quad
H^0(x)+\sum\limits_{i=1}\limits^{\infty}H^i(x),
\]
where $G^i(x)$ and $H^i(x)$ are homogeneous polynomials of degrees
$\deg{G^0(x)}+i$ and $\deg{H^0(x)}+i$, respectively. Then we have
\begin{eqnarray}\label{e3.5}
F(x)=\frac{G(x)}{H(x)}&=&\left(\frac{G^0(x)}{H^0(x)}+\sum\limits_{i=1}
\limits^{\infty}\frac{G^i(x)}{H^0(x)}\right)\left(1+\sum \limits_{i=1}
\limits^{\infty}\frac{H^i(x)}{H^0(x)}\right)^{-1}\nonumber\\
&= & \frac{G^0(x)}{H^0(x)}+\sum\limits_{i=1}\limits^{\infty}\frac{A^i(x)}{B^i(x)},
\end{eqnarray}
where $A^i(x)$ and $B^i(x)$ are homogeneous polynomials. Clearly
\[
\deg{G^0(x)}-\deg{H^0(x)}<\deg {A^i(x)}-\deg{B^i(x)} \qquad \mbox{
for all } i\ge 1.
\]
In what follows we will say that $\deg {A^i(x)}-\deg{B^i(x)}$ is the
degree of $A^i(x)/B^i(x)$, and $G^0(x)/H^0(x)$ is the lowest degree
term of $F(x)$ in the expansion \eqref{e3.5}.  For simplicity we
denote
\[
d(G)=\deg{G^0(x)}, \qquad d(F)=d(G)-d(H)=\deg{G^0(x)}-\deg{H^0(x)},
\]
and call $d(F)$ the {\it lowest degree} of $F$.

\smallskip

\smallskip

\begin{lemma}\label{l3.2} Let
\[
F_1(x)=\frac{G_1(x)}{H_1(x)},\ldots, F_m(x)=\frac{G_m(x)}{H_m(x)},
\]
be functionally independent generalized rational functions in
$(\mathbb C^n,0)$. Then there exist polynomials $P_i(z_1,\ldots,z_m)$
for $i=2,\ldots, m$ such that $F_1(x),\widetilde
F_2(x)=P_2(F_1(x),\ldots,F_m(x)),\ldots,\widetilde
F_m(x)=P_m(F_1(x),\ldots,F_m(x))$ are functionally independent
generalized rational functions, and that $F_1^0(x),\widetilde
F_2^0(x),\ldots,\widetilde F_m^0(x)$ are functionally independent
rational functions.
\end{lemma}

\smallskip

\begin{proof}
This result was first proved by Ziglin \cite{Zi} in 1983, and then proved by Baider {\it et al} \cite{BCRS} in 1996. In order that this paper is self--contained, we provide a proof here (see also the idea of the proof
of Lemma 2.1 of \cite{It89}). 

\smallskip

If $F_1^0(x),\ldots, F_m^0(x)$ are
functionally independent, the proof is done.

\smallskip

Without loss of generality we assume that $F_1^0(x),\ldots, F_k^0(x)$
($1\le k<m$) are functionally independent, and $F_1^0(x),\ldots,
F_{k+1}^0(x)$ are functionally dependent. By Lemma \ref{l3.1} it
follows that $F_1^0(x),\ldots, F_{k+1}^0(x)$ are algebraically
dependent. So there exists a polynomial $P(z)$ of minimal degree with
$z=(z_1,\ldots, z_{k+1})$  such that
\[
P(F_1^0(x),\ldots, F_{k+1}^0(x))\equiv 0.
\]
The fact that $F_1^0(x),\ldots, F_{k}^0(x)$ are algebraically
independent implies
\[
\frac{\partial P}{\partial z_{k+1}}(z)\not\equiv 0.
\]

\smallskip

Since $F_1(x),\ldots,F_m(x)$ are functionally independent, and so
also $F_1(x),\ldots,$ $F_{k+1}(x)$ are functionally independent.
Hence there exists a $(k+1)\times (k+1)$ minor $M=\frac{\partial
(F_1(x),\ldots,F_{k+1}(x))}{\partial (x_{i_1},\ldots,x_{i_{k+1}})}$
of the matrix $\frac{\partial (F_1(x),\ldots,F_{k+1}(x))}{\partial
(x_1,\ldots,x_n)}$ such that its determinant does not vanish. We
denote by $\mathcal D$ this last determinant, and denote by $\mathcal
D^0$ the determinant of the square matrix $M^0:=\frac{\partial
(F_1^0(x),\ldots,F_{k+1}^0(x))}{\partial
(x_{i_1},\ldots,x_{i_{k+1}})}$. We note that $\mathcal D^0$ is the
lowest degree rational homogeneous term of $\mathcal D$.

\smallskip

Define
\[
\mu(F_1,\ldots,F_{k+1})=d(\mathcal
D)+k+1-\sum\limits_{j=1}\limits^{k+1}d(F_j),
\]
where $d(\mathcal D)$ is well--defined, because $\mathcal D$ is also
a generalized rational function. Since  $\partial F_i/\partial
x_j=\left(H_i\frac{\partial G_i}{\partial x_j}-G_i\frac{\partial
H_i}{\partial x_j}\right)/(H_i)^2$ has the lowest  degree larger than
or equal to $d(F_i)-1$ (the former happens if $H_i^0\frac{\partial
G_i^0}{\partial x_j}-G_i^0\frac{\partial H_i^0}{\partial x_j}\equiv
0$), it follows from the definition of $\mathcal D$ that
\[
\mu(F_1,\ldots,F_{k+1})\ge 0.
\]
Furthermore $\mu(F_1,\ldots,F_{k+1})= 0$ if and only if $d(\mathcal
D^0)=d(\mathcal D)$, because $\mathcal D^0$ is the lowest degree
rational homogeneous part of $\mathcal D$. We note that $d(\mathcal
D^0)=d(\mathcal D)$ if and only if $\det(\mathcal D^0)\ne 0$, and
this is equivalent to the functional independence of $F_1^0(x),
\ldots, F_{k+1}^0(x)$. So by the assumption we have
$\mu(F_1,\ldots,F_{k+1})>0$. Set
\[
\widehat F_{k+1}(x)=P(F_1(x),\ldots, F_{k+1}(x)).
\]

\smallskip

First we claim that the functions $F_1(x),\ldots,F_k(x),\widehat
F_{k+1}(x)$ are functionally independent. Indeed, define
\[
\widehat{\mathcal D}:=\det(\partial (F_1(x),\ldots,F_{k}(x),\widehat
F_{k+1}(x))/\partial (x_{i_1},\ldots,x_{i_k},x_{i_{k+1}}).
\]
Then it follows from the functional independence of $F_1(x),\ldots
F_{k+1}(x)$ that
\begin{eqnarray}\label{e3.4}
\widehat{\mathcal D}&=&\det\left(\frac{\partial (F_1(x),\ldots,F_{k}(x),
\widehat F_{k+1}(x))}{\partial(F_1,\ldots, F_k,F_{k+1})}\,
\frac{\partial(F_1,\ldots, F_k, F_{k+1})}{\partial (x_{i_1},\ldots,x_{i_k},
x_{i_{k+1}})}\right)\nonumber\\
&=&\mathcal D\,\frac{\partial P}{\partial z_{k+1}}(F_1,\ldots,F_{k+1}).
\end{eqnarray}
The first equality is obtained by using the derivative of composite
functions, and the second equality follows from the fact that
\[
\det\left(\frac{\partial (F_1(x),\ldots,F_{k}(x),\widehat
F_{k+1}(x))} {\partial(F_1,\ldots,
F_k,F_{k+1})}\right)=\frac{\partial P}{\partial
z_{k+1}}(F_1,\ldots,F_{k+1}).
\]
Since $(\partial P/\partial z_{k+1})(z_1,\ldots,z_{k+1})\not\equiv 0$
and $P$ has the minimal degree such that $P(F_1^0(x),\ldots,
F_{k+1}^0(x))\equiv 0$, we have $(\partial P/\partial
z_{k+1})(F_1^0(x),\ldots, F_{k+1}^0(x))\not\equiv 0$, and so
$\widehat{\mathcal D}\not\equiv 0$. This proves that
$F_1(x),\ldots,F_{k}(x),\widehat F_{k+1}(x)$ are functionally
independent. The claim follows.

\smallskip

Second we claim that
\[
\mu(F_1,\ldots,F_k,\widehat F_{k+1})<\mu(F_1,\ldots,F_{k+1}).
\]
Indeed, writing the polynomial $P$ as the summation
\[
P(z)=\sum\limits_{\alpha}p_{\alpha}z^\alpha,\quad
z^\alpha=z_1^{\alpha_1}\cdot \ldots\cdot
z_{k+1}^{\alpha_{k+1}},\,\,\,
\alpha=(\alpha_1,\ldots,\alpha_{k+1})\in (\mathbb Z^+)^{k+1}.
\]
Define
\[
\nu=\min\{\langle \alpha,\,(d(F_1),\ldots,d(F_{k+1})\rangle:\,
p_\alpha\ne 0,\,\alpha_{k+1}\ne 0\}.
\]
We have $\nu<d(\widehat F_{k+1})$, because $\widehat
F_{k+1}(x)=P(F_1(x),\ldots,F_{k+1}(x))$ contains $F_{k+1}$ and
$P(F_1^0(x),\ldots,F_{k+1}^0(x))\equiv 0$, so the lowest degree part
in the expansion of $P(z)$ must contain $F_{k+1}(x)$. This implies
the lowest degree of $\widehat F(x)$ is larger than that of $P(z)$.
Moreover we get from \eqref{e3.4} and the definition of $\nu$ that
\[
d(\widehat {\mathcal D})=d(\mathcal D)+d\left(\frac{\partial
P}{\partial z_{k+1}}(F_1,\ldots,F_{k+1})\right) =d(\mathcal
D)+\nu-d(F_{k+1}),
\]
where the last equality holds because the partial derivative of $P$
with respect to $z_{k+1}$ is such that $P$ loses one $F_{k+1}$, and
so the total degree loses the degree of $F_{k+1}$. Hence from the
definition of the quantity $\mu$ it follows that
\begin{eqnarray*}
\mu(F_1,\ldots,F_k,\widehat F_{k+1}) &=& d(\widehat{\mathcal
D})+k+1-\sum\limits_{j=1}\limits^{k}d(F_j)-d(\widehat F_{k+1}) \\
&=& \mu(F_1,\ldots,F_{k+1})+\nu-d(\widehat F_{k+1})\\
&<&\mu(F_1,\ldots,F_{k+1}).
\end{eqnarray*}
This proves the claim.

\smallskip

By the two claims, from the functionally independent generalized
rational functions $F_1(x),\ldots, F_k(x),F_{k+1}(x)$ with
$F_1^0(x),\ldots, F_k^0(x),F_{k+1}^0(x)$ being functionally
dependent, we get functionally independent generalized rational
functions $F_1(x),\ldots,F_k(x),\widehat F_{k+1}(x)$, which satisfy
$\mu(F_1,\ldots,F_k,\widehat F_{k+1})$ $<\mu(F_1,\ldots,F_{k+1})$.

\smallskip

If $\mu(F_1,\ldots,F_k,\widehat F_{k+1})=0$, then $F_1^0(x),\ldots,
F_k^0(x),\widehat F_{k+1}^0(x)$ are functionally independent. The
proof is done.

\smallskip

If $\mu(F_1,\ldots,F_k,\widehat F_{k+1})>0$, then $F_1^0(x),\ldots,
F_k^0(x),\widehat F_{k+1}^0$ are also functionally dependent.
Continuing the above the procedure, finally we can get a polynomial
$\widetilde P(z)$  with $z=(z_1,\dots,z_{k+1})$ such that
\[
F_1(x),\ldots,F_k(x),\widetilde F_{k+1}(x)=\widetilde
P(F_1(x),\ldots,F_{k+1}(x)),
\]
are functionally independent and $\mu(F_1,\ldots,F_k,\widetilde
F_{k+1})=0$. The last equality implies that the rational functions
$F_1^0,\ldots,F_k^0,\widetilde F_{k+1}^0$ are functionally
independent. Furthermore, the generalized rational functions
\[
F_1(x),\ldots,F_k(x),\widetilde F_{k+1}(x), F_{k+2}(x),\ldots,F_m(x),
\]
are functionally independent, because $\widetilde F_{k+1}(x)$
involves only $F_1,\ldots,F_{k+1}$, and $F_1,\ldots,F_k,\widetilde
F_{k+1}$ are functionally independent, and also $F_1,\ldots, F_m$ are
functionally independent. This can also be obtained by direct
calculations as follows
\begin{eqnarray*}
&& \det\left(\frac{\partial (F_1(x),\ldots,F_k(x),\widetilde
F_{k+1}(x), F_{k+2}(x),\ldots,F_m(x))}{\partial
(x_1,\ldots,x_n)}\right)\\ &=&
\det\left(\begin{array}{ccc}\frac{\partial F_1}{\partial x_1} &
\ldots  & \frac{\partial F_1}{\partial x_n}\\
\vdots & \vdots & \vdots \\
\frac{\partial F_k}{\partial x_1} & \ldots  & \frac{\partial
F_k}{\partial x_n}\\ \frac{\partial \widetilde P}{\partial
z_1}\frac{\partial F_1}{\partial x_1}+\ldots+\frac{\partial
\widetilde P}{\partial z_{k+1}}\frac{\partial F_{k+1}}{\partial x_1}
& \ldots  & \frac{\partial \widetilde P}{\partial z_1}\frac{\partial
F_1}{\partial x_n}+\ldots+\frac{\partial \widetilde P}{\partial z_{k+1}}
\frac{\partial F_{k+1}}{\partial x_n}\\
\frac{\partial F_{k+1}}{\partial x_1} & \ldots  & \frac{\partial
F_{k+2}}{\partial x_n}\\
\vdots & \vdots & \vdots \\
\frac{\partial F_m}{\partial x_1} & \ldots  & \frac{\partial
F_m}{\partial x_n}
\end{array}\right)\\
&=&\frac{\partial \widetilde P}{\partial
z_{k+1}}(x)\,\det\left(\frac{\partial (F_1(x),\ldots,
F_m(x))}{\partial (x_1,\ldots,x_n)}\right)\ne 0.
\end{eqnarray*}

\smallskip

If $k+1=m$, the proof is completed. Otherwise we can continue the
above procedure, and finally we get the functionally independent
generalized rational functions $F_1(x),\ldots, F_k(x), \widetilde
F_{k+1}(x)=\widetilde P_{k+1}(F_1(x),\ldots,F_{k+1}(x)),\ldots, $
$\widetilde F_m(x)=\widetilde P_m(F_1(x),\ldots,F_m(x))$ such that
their lowest order rational functions $F_1^0(x), F_k^0(x),\widetilde
F_{k+1}^0(x),\ldots,\widetilde F_m^0(x)$ are functionally
independent, where $\widetilde P_j$ for $j=k+1,\ldots, m$ are
polynomials in $F_1,\ldots,F_j$.

\smallskip

The proof of the lemma is completed.
\end{proof}

\smallskip

The next result characterizes rational first integrals of system
\eqref{e1.1}. A {\it rational monomial} is by definition the ratio of
two monomials, i.e. of the form $x^{\mathbf k}/x^{\mathbf l}$ with
$\mathbf k,\mathbf l\in (\mathbb Z^+)^n$. The rational monomial
$x^{\mathbf k}/x^{\mathbf l}$ is {\it resonant} if $\langle
\lambda,\mathbf k-\mathbf l\rangle=0$. A rational function is {\it
homogeneous} if its denominator and numerator are both homogeneous
polynomials. A rational homogeneous function is {\it resonant} if the
ratio of any two elements in the set of all its monomials in both
denominator and numerator is a resonant rational monomial.

\smallskip

In system \eqref{e1.1} we assume without loss of generality that $A$
is in its Jordan normal form and is a lower triangular matrix. Set
$f(x)=Ax+g(x)$ with $g(x)=O(x^2)$. The vector field associated to
\eqref{e1.1} is written in
\[
\mathcal X=\mathcal X_1+\mathcal X_h:=\langle
Ax,\partial_x\rangle+\langle g(x),\partial_x\rangle.
\]

\smallskip

\begin{lemma}\label{l3.3}
If $F(x)=G(x)/H(x)$ is a generalized rational first integral of the
vector field $\mathcal X$ defined by \eqref{e1.1}, then
$F^0(x)=G^0(x)/H^0(x)$ is a resonant rational homogeneous first
integral of the linear vector field $\mathcal X_1$, where we assume
that $F^0$ is non--constant, otherwise if $F^0(x)\equiv a\in \mathbb
C$, then we consider $(F-a)^0$, which is not a constant.
\end{lemma}

\smallskip

For proving this last lemma we will use the following result (see
Lemma 1.1 of \cite{Bi}, for a different proof see for example
\cite{LLZ03}).

\begin{lemma}\label{l3.4}
Let $\mathcal H_n^m$ be the linear space of complex coefficient
homogeneous polynomials of degree $m$ in $n$ variables. For any
constant $c\in\mathbb C$, define a linear operator on $\mathcal
H_n^m$ by
\[
L_c(h)(x)=\langle\partial_xh(x),Ax\rangle-c\,h(x), \qquad h(x)\in
\mathcal H_n^m.
\]
Then the spectrum of $L_c$ is
\[
\{\langle \mathbf k,\lambda\rangle-c:\,\,\mathbf k\in(\mathbb
Z^+)^n,\, |\mathbf k|=k_1+\ldots+k_n=m\},
\]
where $\lambda$ are the eigenvalues of $A$.
\end{lemma}

\smallskip

\noindent{\it Proof of Lemma \ref{l3.3}}.  As in \eqref{e3.5} we
write $F(x)$ in
\[
F(x)=F^0(x)+\sum\limits_{i=1}\limits^\infty F^i(x),
\]
where $F^0(x)$ is the lowest order rational homogeneous function and
$F^i(x)$ for $i\in \mathbb N$ are rational homogeneous functions of
order larger than $F^0(x)$. That $F(x)$ is a first integral in a
neighborhood of $0\in \mathbb C^n$ is equivalent to
\[
\langle \partial_x F(x),f(x)\rangle\equiv 0,\qquad x\in(\mathbb
C^n,0).
\]
Equating the lowest order rational homogeneous functions gives
\begin{equation}\label{e3.6}
\langle \partial_x F^0(x),Ax\rangle\equiv 0,\qquad i.e.
\quad\left\langle \partial_x
\left(\frac{G^0(x)}{H^0(x)}\right),Ax\right\rangle\equiv 0.
\end{equation}
This shows that $F^0(x)$ is a rational homogeneous first integral of
the linear system associated with \eqref{e1.1}.

\smallskip

Next we shall prove that $F^0(x)$ is resonant. {From} the equality
\eqref{e3.6} we can assume without loss of generality that $G^0(x)$
and $H^0(x)$ are relative prime. Now equation \eqref{e3.6} can be
written as
\[
H^0(x)\left\langle
\partial_xG^0(x),Ax\right\rangle-G^0(x)\left\langle
\partial_xH^0(x),Ax\right\rangle\equiv 0.
\]
So there exists a constant $c$ such that
\[
\left\langle \partial_xG^0(x),Ax\right\rangle-cG^0(x)\equiv 0,\quad
\left\langle \partial_xH^0(x),Ax\right\rangle-cH^0(x)\equiv 0.
\]
Set $\deg G^0(x)=l$, $\deg H^0(x)=m$ and $L_c$ be the linear operator
defined in Lemma \ref{l3.4}. Recall from Lemma \ref{l3.4} that $L_c$
has respectively the spectrums on $\mathcal H_n^l$
\[
\mathcal S_l:=\{\langle
\textbf{l},\lambda\rangle-c:\,\,\textbf{l}\in(\mathbb Z^+)^n,\,
|\textbf{l}|=l\},
\]
and on $\mathcal H_n^m$
\[
\mathcal S_m:=\{\langle
\textbf{m},\lambda\rangle-c:\,\,\textbf{m}\in(\mathbb Z^+)^n,\,
|\textbf{m}|=m\}.
\]

\smallskip

Separate $\mathcal H_n^l=\mathcal H_{n1}^l+\mathcal H_{n2}^l$ in such
a way that for any $p(x)\in \mathcal H_{n1}^l $ its monomial
$x^{\mathbf l}$ satisfies $\langle \textbf{l},\lambda\rangle-c=0$,
and for any $q(x)\in \mathcal H_{n2}^l $ its monomial $x^{\mathbf l}$
satisfies $\langle \textbf{l},\lambda\rangle-c\ne 0$. Separate
$G^0(x)$ in two parts $G^0(x)=G_1^0(x)+G_2^0(x)$ with
$G_1^0\in\mathcal H_{n1}^l$ and $G_2^0\in\mathcal H_{n2}^l$. Since
$A$ is in its Jordan normal form and is lower triangular, it follows that
\[
L_c\mathcal H_{n1}^l\subset \mathcal H_{n1}^l, \quad \mbox{and} \quad
L_c\mathcal H_{n2}^l\subset \mathcal H_{n2}^l.
\]
Hence $L_cG^0(x)\equiv 0$ is equivalent to
\[
L_cG_1^0(x)\equiv 0\quad \mbox{and } \quad L_cG_2^0(x)\equiv 0.
\]
Since $L_c$ has the spectrum without zero element on $\mathcal
H_{n2}^l$ and so it is invertible $\mathcal H_{n2}^l$, the equation
$L_cG_2^0(x)\equiv 0$ has only the trivial solution, i.e.
$G_2^0(x)\equiv 0$. This proves that $G^0(x)=G_1^0(x)$, i.e. each
monomial, say  $x^{\mathbf l}$, of $G^0(x)$ satisfies $\langle
{\mathbf l},\lambda\rangle-c=0$.

\smallskip

Similarly we can prove that each monomial, say  $x^{\mathbf m}$, of
$H^0(x)$ satisfies $\langle {\mathbf m},\lambda\rangle-c=0$. This
implies that $\langle {\mathbf l-\mathbf m},\lambda\rangle=0$. The
above proofs show that $F^0(x)=G^0(x)/H^0(x)$ is a resonant rational
homogeneous first integral of $\mathcal X_1$. \hspace{100mm}  $\Box$

\smallskip

Having the above lemmas we can prove Theorem \ref{t1}.

\begin{proof}[Proof of Theorem \ref{t1}]
Let
\[
F_1(x)=\frac{G_1(x)}{H_1(x)},\ldots,F_m(x)=\frac{G_m(x)}{H_m(x)},
\]
be the $m$ functionally independent generalized rational first
integrals of $\mathcal X$. Since the polynomial functions of $F_i(x)$
for $i=1,\ldots,m$ are also generalized rational first integrals of
$\mathcal X$, so by Lemma \ref{l3.2} we can assume without loss of
generality that
\[
F_1^0(x)=\frac{G_1^0(x)}{H_1^0(x)},\ldots,F_m^0(x)=\frac{G_m^0(x)}{H_m^0(x)},
\]
are functionally independent.

\smallskip

Lemma \ref{l3.3} shows that $F_1^0(x),\ldots,F_m^0(x)$ are resonant
rational homogeneous first integrals of the linear vector field
$\mathcal X_1$, that is, these first integrals are rational functions
in the variables given by resonant rational monomials. According to
the linear algebra (see for instance \cite{Bo}), the square matrix
$A$ in $\mathbb C$ has a unique representation in the form
$A=A_s+A_n$ with $A_s$ semi--simple and $A_n$ nilpotent and
$A_sA_n=A_nA_s$. The semi--simple matrix $A_s$ is similar to a
diagonal matrix. Without loss of generality we assume that $A_s$ is
diagonal, i.e. $A_s=\mbox{diag}(\lambda_1,\ldots,\lambda_n)$. Define
$\mathcal X_s=\langle A_sx,\partial_x\rangle$ and $\mathcal
X_n=\langle A_nx,\partial_x\rangle$. Separate $\mathcal X_1=\mathcal
X_s+\mathcal X_n$. Direct calculations show that any resonant
rational monomial is a first integral of $\mathcal X_s$ (for example,
let $x^m$ be a resonant rational monomial, i.e. it satisfies $\langle
\lambda,m\rangle=0$. Then $\mathcal X_s(x^m)=\langle \lambda,m\rangle
x^m=0$). So $F_1^0(x),\ldots,F_m^0(x)$ are also first integrals of
$X_s$. This means that $m$ is less than or equal to the number of
functionally independent resonant rational monomials. In addition, we
can show that the number of functionally independent resonant
rational monomials is equal to the maximum number of linearly independent
vectors in $\mathbb R^n$ of the set $\{\mathbf k\in \mathbb Z^n:\,
\langle \mathbf k,\lambda\rangle=0\}$.

\smallskip

This completes the proof of the theorem. \end{proof}

\section{Proof of Theorem \ref{t4}} \label{s5}

We only consider system \eqref{e1.1} to be positively
semi--quasi--homogeneous. The negative case can be studied similarly,
and its details are omitted. An analytic function $w(x)$ is {\it
semi--quasi--homogeneous of degree $k$} with weight exponent $\mathbf
s$ if $w(x)=w_k(x)+w_h(x)$, where $w_k$ is quasi--homogeneous of
degree $k$ and $w_h$ is the sum of quasi--homogeneous polynomials of
degree larger than $k$.  A rational function $G(x)/H(x)$ is {\it
rational quasi--homogeneous} with weight exponent $\mathbf s$ if
$G(x)$ and $H(x)$ are both quasi--homogeneous with weight exponent
$\mathbf s$. In this section, for an analytic function $w(x)$ we
denote by $w^{(q)}(x)$ its lowest degree quasi--homogeneous part. For
a generalized rational function $F(x)=G(x)/H(x)$ we denote by
$F^{(q)}(x)$ the rational quasi--homogeneous function
$G^{(q)}(x)/H^{(q)}(x)$.

\smallskip

The following result is the key point for proving Theorem \ref{t4},
which is a generalization of Lemma \ref{l3.2} to rational
quasi--homogenous functions. Its proof can be obtained in the same
way as that of Lemma \ref{l3.2}, where we replace the usual degree by
the weight degree. The details are omitted.

\begin{lemma}\label{l5.1}
Let
\[
F_1(x)=\frac{G_1(x)}{H_1(x)},\ldots,F_m(x)=\frac{G_m(x)}{H_m(x)},
\]
be functionally independent generalized rational functions in
$(\mathbb C^n,0)$ with $G_i$ and $H_i$ semi--quasi--homogeneous for
$i=1,\ldots,m$. Then there exist polynomials $P_i(z_1,\ldots,z_m)$
for $i=2,\ldots,m$ such that $F_1(x),\widetilde
F_2(x)=P_2(F_1(x),\ldots,$  $F_m(x)), \ldots,\widetilde
F_m(x)=P_m(F_1(x),\ldots,F_m(x))$ are functionally independent
generalized rational functions, and that $F_1^{(q)}(x),\widetilde
F_2^{(q)}(x),\ldots, \widetilde F_m^{(q)}(x)$ are functionally
independent rational quasi--homogeneous functions.
\end{lemma}

\smallskip

Now we shall prove Theorem \ref{t4}. Since system \eqref{e1.1} is
semi--quasi--homogeneous of degree $q>1$, we take the change of
variables
\[
x\rightarrow \rho^{\mathbf S}x, \quad t\rightarrow \rho^{-(q-1)}t,
\]
where $\rho^{\mathbf
S}=\mbox{diag}\left(\rho^{s_1},\ldots,\rho^{s_n}\right)$. System
\eqref{e1.1} is transformed into
\begin{equation}\label{e5.1}
\dot x=f_q(x)+\widetilde f_h(x,\rho),
\end{equation}
where $\widetilde f_h(x,\rho)=\sum\limits_{i\ge 1}\rho^i\widetilde
f_{q+i}(x)$ and $\rho^{\mathbf E-\mathbf S}\widetilde f_{q+i}(x)$ is
quasi--homogeneous of weight degree $q+i$.

\smallskip

If $F(x)=G(x)/H(x)$ is a generalized rational first integral of
system \eqref{e1.1} with $G(x)$ and $H(x)$ semi--quasi--homogeneous
of weight degree $l$ and $m$ with weight exponent $\mathbf s$
respectively, then
\[
F(x,\rho):=\frac{\rho^m G(\rho^{\mathbf S}x)}{\rho^l H(\rho^{\mathbf
S}x)}=F^{(q)}(x)+\ldots,
\]
is a generalized rational first integral of the
semi--quasi--homogeneous system \eqref{e5.1}, where the dots denote
the sum of the higher order rational quasi--homogeneous functions.
Some easy calculations show that $F^{(q)}(x)$ is a rational
quasi--homogeneous first integral of the quasi--homogeneous system
\begin{equation}\label{e5.2}
\dot x=f_q(x).
\end{equation}

\smallskip

Let $c_0$ be a balance. Taking the change of variable $x=t^{-\mathbf
W}(c_0+u)$, then
\[
F^{(q)}(x)=\frac{t^{-\frac{l}{q-1}}
G^{(q)}(c_0+u)}{t^{-\frac{m}{q-1}}
H^{(q)}(c_0+u)}=u_0^{l-m}F^{(q)}(c_0+u),
\]
where $u_0=t^{-1/(q-1)}$ will be chosen as a new auxiliary variable.
Define $F^{q0}(u_0,u)=u_0^{l-m}F^{(q)}(c_0+u)$. System \eqref{e5.2}
is transformed into
\begin{equation}\label{e5.3}
u'=Ku+\overline f_q(u),
\end{equation}
where the prime denotes the derivative with respect to $\tau=\ln t$
and
\[
\overline f_q(u)=\mathbf Wc_0+f_q(c_0+u)-\partial_x f_q(c_0) u.
\]

\smallskip

We claim that $F^{q0}(u_0,u)$ is a first integral of
\begin{equation}\label{e5.4}
u_0'=-\frac{1}{q-1}u_0,\quad u'=Ku+\overline f_q(u).
\end{equation}
Indeed, since
$\left\langle\partial_xF^{(q)}(x),f_q(x)\right\rangle\equiv 0$ and
$\mathbf W=\mathbf S/(q-1)$ we have
\begin{eqnarray*}
&&\left.\frac{F^{q0}(u_0,u)}{d\tau}\right |_{\eqref{e5.4}}\\
&&=-\frac{l-m}{q-1}u_0^{l-m}F^{(q)}(c_0+u)+u_0^{l-m}
\left\langle\partial_uF^{(q)}(c_0+u),Ku+\overline f_q(u)\right\rangle\\
&&=-\frac{u_0^{l-m}}{q-1}\left((l-m)F^{(q)}(c_0+u)-\langle
\partial_uF^{(q)}(c_0+u),\mathbf
S(c_0+u)\rangle\right)\\ &&=0,
\end{eqnarray*}
where we have used the facts that
$F^{(q)}(c_0+u)=G^{(q)}(c_0+u)/H^{(q)}(c_0+u)$, and the generalized
Euler's formula:
\[
\left\langle\partial_xG^{(q)}(x),\mathbf S
x\right\rangle=lG^{(q)}(x)\,\, \mbox{ and }\,\,
\left\langle\partial_xH^{(q)}(x),\mathbf S
x\right\rangle=mH^{(q)}(x),
\]
because $G^{(q)}(x)$ and $H^{(q)}(x)$ are quasi--homogeneous
polynomials of weight degree $l$ and $m$ with weight exponents
$\mathbf s$, respectively. The claim follows.

\smallskip

Assume that system \eqref{e1.1} has the maximal number, say $r$, of
functionally independent generalized rational first integrals
\[
F_1(x)=\frac{G_1(x)}{H_1(x)},\ldots,F_r(x)=\frac{G_r(x)}{H_r(x)}.
\]
By Lemma \ref{l5.1} we can assume that
\[
F_1^{(q)}(x)=\frac{G_1^{(q)}(x)}{H_1^{(q)}(x)},\ldots,F_r^{(q)}
(x)=\frac{G_r^{(q)}(x)}{H_r^{(q)}(x)},
\]
are functionally independent, and they are rational
quasi--homogeneous first integrals of \eqref{e5.2}.

\smallskip

Assume that $G_i^{(q)}(x)$ and $H_i^{(q)}(x)$, $i=1,\ldots,r$, have
weight degree $l_i$ and $m_i$, respectively. Then it follows from the
last claim that
\[
F_1^{q0}(u_0,u)=u_0^{l_1-m_1}F_1^{(q)}(c_0+u), \ldots,
F_r^{q0}(u_0,u)=u_0^{l_r-m_r}F_r^{(q)}(c_0+u),
\]
are functionally independent rational quasi--homogeneous first
integrals of \eqref{e5.4}.

\smallskip

The linear part of \eqref{e5.4} has the eigenvalues
$-1/(q-1),\lambda_{c_0}=(\lambda_1^0,\ldots,\lambda_n^0)$. Since the
Kowalevskaya matrix has always the eigenvalue $-1$ (see for instance
\cite{Zi}), we set $\lambda_1^0=-1$. Then for $\mathbf
k=(k_1,\ldots,k_n)\in\mathbb Z^n$ we have
\[
-k_0+(q-1)\langle\lambda_{c_0},\mathbf k\rangle=\langle
\lambda_{c_0}, (k_0+(q-1)k_1,(q-1)k_2,\ldots,(q-1)k_n)\rangle.
\]
Recall that $d_{c_0}$ is the dimension of the minimum linear subspace
of $\mathbb R^n$ containing the set $\{\mathbf k\in\mathbb Z^n:\,
\langle\lambda_{c_0},\mathbf k\rangle=0\}$. Hence we get from Theorem
\ref{t1} that equation \eqref{e5.4} has at most $d_{c_0}$
functionally independent first integrals. These proofs imply that
$r\le d_{c_0}$, and consequently $r\le d=\min\limits_{c\in\mathcal
B}d_c$.

\smallskip

This completes the proof of the theorem.

\section{Proof of Theorem \ref{t5}}\label{s6}

Assume that system \eqref{e1.1} has the maximal number, say $r$, of
functionally independent generalized rational first integrals in a
neighborhood $U$ of the given periodic orbit, denoted by
\[
F_1(x)=\frac{G_1(x)}{H_1(x)},\ldots, F_r(x)=\frac{G_r(x)}{H_r(x)},
\]
where $G_i$ and $H_i$ are analytic functions. Let $P(x)$ be the
Poincar\'e map defined in a neighborhood of the periodic orbit. Then
$F_i(x)$ for $i=1,\ldots,r$ are also first integrals of $P(x)$.
Recall that a continuous function $C(x)$ is a {\it first integral} of
a homeomorphism $M(x)$ defined in an open subset $U$ of $\mathbb C^n$
if $C(M^m(x))=C(x)$ for all $m\in\mathbb Z$ and $x\in U$.

\smallskip

We now turn to study the maximal number of functionally independent
generalized rational first integrals of the Poincar\'e map. Since a
polynomial function of generalized rational first integrals of a map
is also a generalized rational first integral of the map, so by Lemma
\ref{l3.2} we can assume without loss of generality that
$F_1^0(x),\ldots,F_r^0(x)$ are functionally independent.

Since system \eqref{e1.1} is analytic, the Poincar\'e map $P(x)$ is
analytic (see for instance, \cite{Ch}). Set
\begin{equation}\label{e6.1}
P(x)=Bx+P_h(x),
\end{equation}
where $P_h(x)$ is the higher order terms of $P(x)$. We can assume
without loss of generality that $B$ is in its lower triangular Jordan
normal form. Let $B_s$ be the semi--simple part of $B$. Since
$F_i(x)$ for $i=1,\ldots,r$, are first integrals of $P(x)$, we have
\begin{equation}\label{e6.2}
\frac{G_i(P(x))}{H_i(P(x))}\equiv \frac{G_i(x)}{H_i(x)},\qquad x\in
U.
\end{equation}
As in \eqref{e3.5} we expand $G_i(P(x))/H_i(P(x))$ via \eqref{e6.1}
as the sum of rational homogeneous functions, and equating the lowest
order rational homogeneous terms of \eqref{e6.2}, we get that
\begin{equation}\label{e6.3}
\frac{G_i^0(Bx)}{H_i^0(Bx)}\equiv \frac{G_i^0(x)}{H_i^0(x)},\qquad
x\in U.
\end{equation}
{From} the last equality, we can assume without loss of generality
that $G_i^0(x)$ and $H_i^0(x)$ are relatively prime. Recall that
$A^0(x)$ is the lowest order homogeneous polynomial of a series
$A(x)$. Equation \eqref{e6.3} can be written as
\begin{equation}\label{e6.4}
\frac{G_i^0(Bx)}{G_i^0(x)}\equiv \frac{H_i^0(Bx)}{H_i^0(x)},\qquad
x\in U.
\end{equation}
Since $G_i^0(Bx)$ is either identically zero or a homogeneous
polynomial of the same degree than $G_i^0(x)$, and $G_i^0(x)$ and
$H_i^0(x)$ are relative prime, there exists a constant $c_i$ such
that
\begin{equation}\label{e6.5}
G_i^0(Bx)\equiv c_i G_i^0(x)\quad \mbox{and} \quad H_i^0(Bx)=c_i
H_i^0(x),\qquad x\in U.
\end{equation}

\smallskip

For completing the proof of Theorem \ref{t5} we need the following
result (see for instance, Lemma 11 of \cite{LLZ03}).

\begin{lemma}\label{l6.1}
Let $\mathcal H_n^m$ be the complex linear space of homogeneous
polynomials of degree $m$ in $n$ variables, and let $\mu$ be the
$n$--tuple of eigenvalues of $B$. Define the linear operator from
$\mathcal H_n^m$ into itself by
\[
L_c(h)(x)=h(Bx)-ch(x),\qquad h(x)\in\mathcal H_n^m.
\]
Then the set of eigenvalues of $L_c$ is $\{\mu^{\mathbf k}-c:\,
\mathbf k\in(\mathbb Z^+)^n,\,|\mathbf k|=m\}$.
\end{lemma}

\smallskip

Assume that $G_i^0(x)$ and $H_i^0(x)$ have respectively degrees $l_i$
and $m_i$. Using the notations given in Lemma \ref{l6.1} we write
equations \eqref{e6.5} as
\[
L_{c_i}(G_i^0)(x)\equiv 0 \quad \mbox{and} \quad L_{c_i}(H_i^0)(x)\equiv 0.
\]
Working in a similar way as in the proof of Theorem \ref{t1} we
obtain from Lemma \ref{l6.1} that each monomial, say $x^{\mathbf l}$,
of $G_i^0(x)$  satisfies $\mu^{\mathbf l}-c_i=0$, where $\mathbf l\in
(\mathbb Z^+)^n$ and $|\mathbf l|=l_i$, and that  each monomial, say
$x^{\mathbf m}$, of $H_i^0(x)$ satisfies $\mu^{\mathbf m}-c_i=0$,
where $\mathbf m\in (\mathbb Z^+)^n$ and $|\mathbf m|=m_i$.

\smallskip

The above proof shows that the ratio of any two monomials in the
numerator and denominator of $F_i^0(x)$, $i\in\{1,\ldots,r\}$, is a
resonant monomial. Hence each of the ratios is a first integral of
the vector field $B_sx$, and consequently $F_i^0(x)$, $i=1,\ldots,r$,
are first integrals of $B_sx$. In addition, we can check easily that
the maximal number of functionally independent elements of
$\{x^{\mathbf k}:\,\mu^{\mathbf k}=1,\,\mathbf k\in \mathbb
Z^n,\,\mathbf k\ne 0\}$ is equal to the dimension of the minimal
linear subspace in $\mathbb R^n$ containing the set  $\{\mathbf k\in
\mathbb Z^n:\,\,\mu^{\mathbf k}=1\}$. This proves the theorem.

\section{Proof of Theorem \ref{t3}} \label{s4}

For proving Theorem \ref{t3} we need the following result, which is a
modification of Lemma \ref{l3.2}. Its proof can be got in  the same
way as the proof of Lemma \ref{l3.2}, where we replace the field
$\mathbb C$ by $\widetilde{\mathbb C}(t)$ the field of complex
coefficient generalized rational functions in $t$. The details are
omitted.

\begin{lemma}\label{l4.1}
Let
\[
F_1(t,x)=\frac{G_1(t,x)}{H_1(t,x)},\ldots,F_m(t,x)=\frac{G_m(t,x)}
{H_m(t,x)},\qquad (t,x)\in \mathbb S^1\times (\mathbb C^n,0),
\]
be functionally independent generalized rational functions and $2\pi$
periodic in $t$. Then there exist polynomials $P_i(z_1,\ldots,z_m)$
for $i=2,\ldots,m$ such that $F_1(t,x),\widetilde
F_2(t,x)=P_2(F_1(t,x),\ldots,F_m(t,x)), \ldots,\widetilde F_m(t,x)=$
$P_m(F_1(t,x),\ldots,$ $F_m(t,x))$ are functionally independent
generalized rational functions, and that $F_1^0(t,x),\widetilde
F_2^0(t,x),\ldots,$   $\widetilde F_m^0(t,x)$ are functionally
independent rational homogeneous functions.
\end{lemma}

\smallskip

In the proof of Theorem \ref{t3} we also need the Floquet's Theorem.
For readers' convenience we state it here.

\smallskip

\noindent{\bf Floquet's Theorem} {\it There exists a change of
variables $x=B(t)y$ periodic of period $2\pi$ in $t$, which
transforms the linear periodic differential system \eqref{e1.4} into
the linear autonomous one
\[
\dot y=\Lambda y,\qquad \Lambda \mbox{ is a constant matrix}.
\]
Furthermore the characteristic multipliers $\mu$ of \eqref{e1.4}
satisfy $\mu_i=\exp(2\pi\lambda_i)$ for $i=1,\ldots,n$, where
$\lambda=(\lambda_1,\ldots,\lambda_n)$ are the eigenvalues of
$\Lambda$.}

\smallskip

Now we can prove Theorem \ref{t3}. Assume that system \eqref{e1.3}
has the maximal number, say $m$, of functionally independent
generalized rational first integrals, denoted by
\[
F_1(t,x)=\frac{G_1(t,x)}{H_1(t,x)},\ldots,
F_m(t,x)=\frac{G_m(t,x)}{H_m(t,x)}.
\]
By Lemma \ref{l4.1} we can assume without loss of generality that the
rational homogeneous functions
\[
F_1^0(t,x)=\frac{G_1^0(t,x)}{H_1^0(t,x)},\ldots,
F_m^0(t,x)=\frac{G_m^0(t,x)}{H_m^0(t,x)},
\]
are functionally independent.

\smallskip

By the Floquet's Theorem, system \eqref{e1.3} is transformed via a
change of the form $x=B(t)y$ into
\begin{equation}\label{e4.1}
\dot y=\Lambda y+h(t,y),
\end{equation}
where $h(t,y)=O(y^2)$ is $2\pi$ periodic in $t$. Therefore system
\eqref{e4.1} has the functionally independent generalized rational
first integrals
\[
\widetilde F_1(t,y)=\frac{\widetilde G_1(t,y)}{\widetilde H_1(t,x)}=
\frac{G_1(t,B(t)y)}{H_1(t,B(t)y)},\ldots, \widetilde
F_m(t,y)=\frac{\widetilde G_m(t,y)}{\widetilde
H_m(t,x)}=\frac{G_m(t,B(t)y)}{H_m(t,B(t)y)}.
\]
and
\[
\widetilde F_1^0(t,y)=\frac{\widetilde G_1^0(t,y)}{\widetilde
H_1^0(t,x)}, \ldots, \widetilde F_m^0(t,y)=\frac{\widetilde
G_m^0(t,y)}{\widetilde H_m^0(t,x)},
\]
are also functionally independent. We can assume without loss of
generality that $\widetilde G_i^0(t,y)$ and $\widetilde H_i^0(t,x)$
have respectively degrees $l_i$ and $m_i$, and  are relatively prime
for $i=1,\ldots,m$.

\smallskip

We expand $\widetilde F_i(t,y)$, $i=1,\ldots,m$, in the way done in
\eqref{e3.5}, and since $\widetilde F_i(t,y)$ are first integrals of
\eqref{e4.1}, we get that
\begin{equation}\label{e4.2}
\partial_t \widetilde F_i^0(t,y)+\left\langle\partial_y \widetilde
F_i^0(t,y),  \Lambda y\right\rangle\equiv 0,\qquad i=1,\ldots,m,
\end{equation}
i.e., $\widetilde F_i^0(t,y)$, $i=1,\ldots,m$, are functionally
independent first integrals of the linear differential system
\begin{equation}\label{e4.3}
\dot y=\Lambda y.
\end{equation}
Equations \eqref{e4.2} are equivalent to
\begin{equation}\label{e4.4}
\begin{array}{l}
\widetilde H_i^0(t,y)\left(\partial_t \widetilde
G_i^0(t,y)+\left\langle\partial_y \widetilde G_i^0(t,y),  \Lambda
y\right\rangle\right)\\ \equiv \widetilde G_i^0(t,y)\left(\partial_t
\widetilde H_i^0(t,y)+\left\langle\partial_y \widetilde H_i^0(t,y),
\Lambda y\right\rangle\right),\qquad i=1,\ldots,m.
\end{array}
\end{equation}
So there exist constants, say $c_i$, such that
\begin{equation}\label{e4.5}
\partial_t \widetilde G_i^0(t,y)+\left\langle\partial_y \widetilde
G_i^0(t,y),  \Lambda y\right\rangle-c_i \widetilde G_i^0(t,y)\equiv
0,
\end{equation}
and
\begin{equation}\label{e4.6}
\partial_t \widetilde H_i^0(t,y)+\left\langle\partial_y \widetilde
H_i^0(t,y),  \Lambda y\right\rangle-c_i \widetilde H_i^0(t,y)\equiv
0.
\end{equation}

For the set of monomials of degree $k$,
$\Upsilon_k:=\left\{y^{\mathbf k}:\,\mathbf k\in \left(\mathbb
Z^+\right)^n,\,|\mathbf k|=k\right\}$, we define their order as
follows:  $y^{\mathbf p}$ is before $y^{\mathbf q}$ if $\mathbf
p-\mathbf q\succ 0$, i.e., there exists an $i_0\in\{1,\ldots,n\}$
such that $p_i=q_i$ for $i=1,\ldots,i_0-1$ and $p_{i_0}>q_{i_0}$.
Then $\Upsilon_k$ is a base of the set of homogeneous polynomials of
degree $k$ with the given order. According to the given base and
order, each homogeneous polynomial of degree $k$ is uniquely
determined by its coefficients.

\smallskip

We denote by $\widetilde G_i^0(t)$ the vector of dimension
$\left(\begin{array}{c}l_i+n-1\\ n-1\end{array}\right)$ formed by the
coefficients of $\widetilde G_i^0(t,y)$. Let $\mathcal L_k$ be the
linear operator on $\mathcal H_n^k(t)$, the linear space of
homogeneous polynomials of degree $k$ in $y$ with coefficients $2\pi$
periodic in $t$, defined by
\[
\mathcal L_k(h(t,y))=\langle \partial_y h(t,y),\Lambda y\rangle,
\qquad h(t,y)\in\mathcal H_n^k(t).
\]
Using these notations equations \eqref{e4.5} and \eqref{e4.6} can be
written as
\[
\partial_t\widetilde G_i^0(t)+(\mathcal L_{l_i}-c_i)\widetilde G_i^0(t)\equiv 0, \quad
\partial_t\widetilde H_i^0(t)+(\mathcal L_{m_i}-c_i)\widetilde H_i^0(t)\equiv 0.
\]
They have solutions
\[
\widetilde G_i^0(t)=\exp\left((c_i\mathbf E_{1i}-\mathcal
L_{l_i})t\right) \widetilde G_i^0(0), \quad \widetilde
H_i^0(t)=\exp\left((c_i\mathbf E_{2i}-\mathcal
L_{m_i})t\right)\widetilde H_i^0(0),
\]
where $\mathbf E_{1i}$ and $\mathbf E_{2i}$ are two identity matrices
of suitable orders. In order that $\widetilde G_i^0(t)$ and
$\widetilde H_i^0(t)$ be $2\pi$ periodic, we should have
\begin{equation}\label{e4.7}
\left(\exp\left((c_i\mathbf E_{1i}-\mathcal
L_{l_i})2\pi\right)-\mathbf E_{1i}\right)\widetilde G_i^0(0)=0,
\end{equation}
and
\begin{equation}\label{e4.8}
\left(\exp\left((c_i\mathbf E_{2i}-\mathcal
L_{m_i})2\pi\right)-\mathbf E_{2i}\right)\widetilde H_i^0(0)=0.
\end{equation}

\smallskip

Recall that $\lambda=(\lambda_1,\ldots,\lambda_n)$ be the eigenvalues
of $\Lambda$. Then it follows from Lemma \ref{l3.4} that
$\exp((c_i\mathbf E_{2i}-\mathcal L_{l_i})2\pi)$ and
$\exp((c_i\mathbf E_{2i}-\mathcal L_{m_i})2\pi)$ have respectively
the eigenvalues
\[
\left\{\exp((c_i-\langle \mathbf l,\lambda\rangle)2\pi):\, \mathbf
l\in (\mathbb Z^+)^n,\, |\mathbf l|=l_i\right\},
\]
and
\[
\left\{\exp((c_i-\langle \mathbf m,\lambda\rangle)2\pi):\, \mathbf
m\in (\mathbb Z^+)^n,\, |\mathbf m|=m_i\right\}.
\]
In order that equations \eqref{e4.7} and \eqref{e4.8} have nontrivial
solutions we must have
\[
\exp(\langle \mathbf l,\lambda\rangle 2\pi)=\exp(c_i2\pi) \quad
\mbox{ for all } \mathbf l\in (\mathbb Z^+)^n,\, |\mathbf l|=l_i,
\]
and
\[
\exp(\langle \mathbf m,\lambda\rangle 2\pi)=\exp(c_i2\pi) \quad
\mbox{ for all } \mathbf m\in (\mathbb Z^+)^n,\, |\mathbf m|=m_i.
\]
It follows that
\[
\exp(\langle \mathbf l-\mathbf m,\lambda\rangle 2\pi)=1 \quad \mbox{
for all } {\bf l, m}\in (\mathbb Z^+)^n,\, |\mathbf l|=l_i,\,
|\mathbf m|=m_i,
\]
i.e.,
\[
\mu^{\mathbf l-\mathbf m}=1 \quad \mbox{ for }\,\, {\bf l, m}\in
(\mathbb Z^+)^n,\, |\mathbf l|=l_i,\, |\mathbf m|=m_i.
\]

\smallskip

The above proof shows that the ratio of any two monomials in the
denominator and numerator of each $\widetilde F_i^0(t,y)$ for
$i\in\{1,\ldots,m\}$ is resonant. Hence working in a  similar way to
the proof of Theorem \ref{t1} we get that $m$ is at most the
dimension of $\Xi$.

\smallskip

This completes the proof of the theorem.

\section*{Acknowledgements}

We sincerely thank the referee for his/her valuable suggestions and comments in which who mentioned 
the existence of a proof on Lemma \ref{l3.2} given in \cite{Zi} and \cite{BCRS}.

\smallskip

The second author is partially supported by a MCYT/FEDER grant number
MTM 2008-03437, by a CICYT grant number 2009SGR 410 and by ICREA
Academia. The third author is partially supported by NNSF of China
grant 10831003 and  Shanghai Pujiang Programm 09PJD013.

\end{document}